\documentclass[a4paper,11pt]{amsart}

\pdfoutput=1

\usepackage[T1]{fontenc}
\usepackage{amssymb,amsfonts,amsmath,amsthm}
\usepackage{url}
\usepackage{color}
\usepackage{enumerate}
\usepackage{pinlabel}

\usepackage{mathabx}

\usepackage[top=1.2in,bottom=1.4in,left=1.4in,right=1.4in]{geometry}

\input{xy}
\xyoption{all}
\objectmargin={3mm}

\newcounter{notes}
\newcommand{\ignore}[1]{}

\newtheorem{theorem}{Theorem}
\newtheorem{proposition}[theorem]{Proposition}

\newtheorem{lemma}[theorem]{Lemma}

\newtheorem{observation}[theorem]{Observation}

\theoremstyle{definition}
\newtheorem{definition}[theorem]{Definition}

\newtheoremstyle{theoremwithref}{}{}{\itshape}{}{\bfseries}{.}{.5em}{#1 #2 #3}
\theoremstyle{theoremwithref}

\newcommand{\R}{\mathbb{R}}

\newcommand{\Z}{\mathbb{Z}}

\newcommand{\SU}{\mathrm{SU}}

\newcommand{\App}{\mathrm{App}}

\newcommand{\Mod}{\mathrm{Mod}}
\newcommand{\tr}{\mathrm{tr}}

\newcommand{\Ind}{\mathrm{Ind}}

\title[Transitivity of normal subgroups of $\Mod(\Sigma_g)$ on
character varieties]{Transitivity of normal subgroups of the mapping class groups
on character varieties}
\author{Julien March\'e}
\address{Sorbonne Universit\'es, UPMC Univ.\ Paris 06, Institut de Math\'ematiques
de Jussieu-Paris Rive Gauche, UMR 7586, CNRS, Univ. Paris Diderot, Sorbonne
Paris Cit\'e, 75005 Paris, France}
\email{julien.marche@imj-prg.fr}

\author{Maxime Wolff}
\address{Sorbonne Universit\'es, UPMC Univ.\ Paris 06, Institut de Math\'ematiques
de Jussieu-Paris Rive Gauche, UMR 7586, CNRS, Univ. Paris Diderot, Sorbonne
Paris Cit\'e, 75005 Paris, France}
\email{maxime.wolff@imj-prg.fr}

\begin{document}

\maketitle

\begin{abstract}
  We prove that the action of any non-trivial normal subgroup of the
  mapping class group of a surface of genus $g\geqslant 2$
  is almost minimal
  on the character variety
  $X(\pi_1\Sigma_g,\SU_2)$: the orbit of almost every point is
  dense.
\vspace{0.2cm}

\end{abstract}

\section{Introduction}\label{intro}

For every $g\geqslant 2$, let $\pi_1\Sigma_g$ denote a fundamental
group of a compact, connected, orientable surface of genus $g$,
and $\Mod(\Sigma_g)$ its mapping class group.
In \cite{Goldman97}, Goldman proved that $\Mod(\Sigma_g)$ acts
ergodically on the character variety $X(\pi_1\Sigma_g,\SU_2)$,
and subsequently, Previte and Xia \cite{PreviteXia}
proved that for every
conjugacy class of representation $\rho\colon\pi_1\Sigma_g\to\SU_2$
with dense image, the orbit $\Mod(\Sigma_g)\cdot[\rho]$ is
dense in $X(\pi_1\Sigma_g,\SU_2)$.

Goldman then raised (see \cite{Goldman05}) the question of whether
smaller subgroups of $\Mod(\Sigma_g)$ still act ergodically on
$X(\pi_1\Sigma_g,\SU_2)$, and with Xia he proved~\cite{GoldmanXia12}
that when $\Sigma$ is a twice punctured torus, the Torelli group
acts ergodically in the relative $\SU_2$ character varieties.
This question was addressed by
Funar and March\'e~\cite{FunarMarche}, who proved that the
Johnson subgroup, generated by the Dehn twists along separating curves,
acts ergodically on this character variety. Provided $g\geqslant 3$,
Bouilly (see~\cite{Bouilly})
gave a simpler proof that the
Torelli group acts ergodically on this character variety, and
in fact on the topological components of the character variety
$X(\pi_1\Sigma_g, G)$ for any compact Lie group $G$.

In this note, when a group $\Gamma$ acts on a topological space
$X$ endowed with a Radon measure $\mu$, we will say that the action
is {\em almost minimal} if the orbit of almost every point is dense.
We say the action is minimal if every orbit is dense, and ergodic
if for every measurable $\Gamma$-invariant set $U$, either $U$ or
its complement has measure~$0$. These two latter properties are
independent in general, while both imply almost minimality.

The main result of this note is the following.
\begin{theorem}\label{thm:PresqueMin}
  Suppose $g\geqslant 2$.
  Let $\Gamma$ be a non central, normal subgroup of $\Mod(\Sigma_g)$.
  Then the action of $\Gamma$ on $X(\pi_1\Sigma_g,\SU_2)$ is
  almost minimal.
\end{theorem}

When $g\geqslant 3$, the centre of $\Mod(\Sigma_g)$ is trivial,
while if $g=2$ this centre is isomorphic to $\Z/2\Z$, and generated
by the hyperelliptic involution.
The hypothesis ``non central'' simply rules out the cases when
$\Gamma$ is trivial or equal to this central $\Z/2\Z$ subgroup.
Thus Theorem~\ref{thm:PresqueMin} applies,
for example,
to every
term of the lower central series of $\Mod(\Sigma_g)$.

The mapping class group $\Mod(\Sigma_g)$ is generated by Dehn twists,
while the Torelli group is generated by products of the form
$\tau_\gamma\tau_\delta^{-1}$ where $(\gamma,\delta)$ is a pair of
cobordant simple curves.
Bouilly's approach to the ergodicity of the Torelli group uses
the idea that, for almost every conjugacy class of representation
$[\rho]$ and for every bounding pair $(\gamma,\delta)$, the
product
$\tau_\gamma\tau_\delta^{-1}$ acts as a
totally irrational rotation along a torus immersed in the character
variety $X(\pi_1\Sigma_g,\SU_2)$. Thus, for an appropriate sequence
of powers,
$\tau_\gamma^N\tau_\delta^{-N}$ approximates the effect of the
Dehn twist $\tau_\gamma$. This reduces the ergodicity properties
of the Torelli group to those of the whole mapping class group,
and these are well understood.

The key lemma in the proof of Theorem~\ref{thm:PresqueMin}
is Lemma~\ref{lem:PresqueMin} below. 
It consists in extending Bouilly's trick to the case when $\gamma$ and $\delta$ 
are no longer disjoint. We manage to control the action of $\tau_\gamma^n\tau_\delta^{-n}$ 
for some sequences of integers $n$ dictated by classical theorems in Diophantine approximation theory. 



\section{Proof of Theorem~\ref{thm:PresqueMin}}

We first set up some notation.
\subsection{Notation and reminders}
The space $\mathrm{Hom}(\pi_1\Sigma_g,\SU_2)$ of morphisms from
$\pi_1\Sigma_g$ to $\SU_2$ is naturally endowed with the product topology
and the {\em character variety} $X(\pi_1\Sigma_g,\SU_2)$ is the
quotient of this representation space by the conjugation action
of~$\SU_2$. From now on we will denote it simply by $X$.

The mapping class group
$\Mod(\Sigma_g)=\pi_0(\mathrm{Diff}_+(\Sigma_g))$
is, by the Dehn-Nielsen-Baer theorem, isomorphic to an index two
subgroup of $\mathrm{Out}(\pi_1\Sigma_g)$. It acts naturally on $X$,
by setting, for $\phi\in\mathrm{Aut}(\Sigma_g)$ and $[\rho]\in X$,
$\phi\cdot [\rho] = [\rho\circ\phi^{-1}]$: this descends to an action
of $\mathrm{Out}(\pi_1\Sigma_g)$.

The mapping class group is generated by the Dehn twists:
when $\gamma\subset\Sigma$ is a simple closed curve, we denote by
$\tau_\gamma$ the Dehn twist along $\gamma$;
see~\textsl{e.g.} \cite[Chapter~3]{FarbMargalit} for a definition,
and numerous properties.
Given such a curve,
we may choose a representant in $\pi_1\Sigma_g$:
such a representant is well defined up to conjugacy and up to
passing to the inverse. Yet, we will often use the same notation,
$\gamma$ for the corresponding elements of $\pi_1\Sigma_g$.

For every element $A\in\SU_2$, we will write
$\theta(A)=\frac{1}{\pi}\arccos(\frac{1}{2}\tr(A))\in[0,1]$.
Note that this is also invariant by conjugation and by taking the
inverse. Thus, when $\gamma$ is an unoriented closed curve, or
an element of $\pi_1\Sigma_g$,
we also define 
$\theta_\gamma\colon X\to [0,1]$ by
$\theta_\gamma([\rho])=\theta(\rho(\gamma)).$
This function is continuous, and smooth on $\theta_\gamma^{-1}((0,1))$.


It is well-known that the subspace of irreducible representations in $X$ forms 
a Zariski open subset $X^{\rm irr}$, which is the smooth part of $X$. 
Moreover, there is a $\Mod(\Sigma_g)$-invariant symplectic form on $X^{\rm irr}$ and 
the Hamiltonian flow of $\theta_\gamma$ on $X^{\rm irr}\cap \theta_\gamma^{-1}((0,1))$, 
denoted  by $\Phi_\gamma^t$ is $1$-periodic. This flow can be extended to $\theta_\gamma^{-1}((0,1))$
 and it satisfies the crucial identity $\tau_\gamma([\rho])=\Phi_\gamma^{\theta_\gamma([\rho])}([\rho])$ for 
all $[\rho]\in \theta_\gamma^{-1}((0,1))$. We refer to \cite{Goldman86} for all these facts.

\subsection{Simultaneous Diophantine approximation}
In the following definition, and subsequently in this note,
for all $x\in\R/\Z$ we will denote by $|x|$ its distance to $0$
in~$\R/\Z$.
\begin{definition}\label{def:App}
  A
  pair $(x,y)$ of irrational elements of $\R/\Z$ will be said
  {\em appropriately approximable} if there exists a
  strictly increasing sequence
  $(q_n)$ of integers such that $q_n x$ converges to $0$
  faster than $\frac{1}{q_n}$
  (\textsl{i.e.}, $|q_n x|=o(\frac{1}{q_n})$)
  and $q_n y$ converges to $y$ in~$\R/\Z$.
\end{definition}
A classical theorem of Khinchin~\cite{Khinchin}
states that if $(\psi_n)$ is a decreasing sequence of real
numbers and if $\sum\psi_n$ diverges, then for almost every
$x$ there are infinitely many integers $q$ such that
$|q x|\leqslant\psi_q$.
In particular for example, for almost every $x$, there are
infinitely many integers $q$ satisfying
$|q x|\leqslant\frac{1}{q\ln q}$.

Now, a classical theorem of Hardy and
Littlewood~\cite[Theorem 1.40]{HardyLittlewood} states that
for every strictly increasing sequence of integers $(q_n)$,
for almost every $y\in\R/\Z$ the set $\{q_n y, n\geqslant 0\}$
is dense in $\R/\Z$. In particular, for almost every $y$,
the number $y$ is an accumulation point of the sequence
$(q_n y)$.

These two theorems together imply the following
observation.
\begin{observation}
  The set $\App\subset (\R/\Z)^2$ of appropriately approximable
  pairs has full measure.
\end{observation}

We continue with some preliminary observations concerning
mapping class groups and character varieties.
\subsection{Preliminary observations}
In the next statements, we denote by $P$ the set of
pairs $(\gamma,\delta)$ of isotopy classes of non-separating and 
non-isotopic simple curves.
\begin{observation}\label{obs:OnBouge}
  Let $\gamma\subset\Sigma_g$ be an unoriented,
  non-separating simple closed curve.
  Then there exists $\varphi\in\Gamma$ such that
  $(\gamma,\varphi(\gamma))\in P$.
\end{observation}
\begin{proof}
  Since $\Gamma$ is not central, and since $\Mod(\Sigma_g)$
  is generated by Dehn twists along non-separating curves,
  there exists a non-separating simple closed curve
  $\delta$, and $\psi\in\Gamma$, such that $\psi$ and
  the Dehn twist $\tau_\delta$ do not commute.
  There exists $\phi\in\Mod(\Sigma_g)$ mapping $\delta$ to
  $\gamma$, so $\phi\tau_\delta\phi^{-1}=\tau_\gamma$.
  Now $\varphi=\phi\psi\phi^{-1}$ is in $\Gamma$ since
  $\Gamma$ is normal, and $\varphi$ does not commute with
  $\tau_\gamma$; this implies the statement.
\end{proof}

For every $(\gamma,\delta)\in P$, we denote by
$\Ind(\gamma,\delta)$ the subset of $X$
consisting of those $[\rho]$ such that
$(\theta(\rho(\gamma)), \theta(\rho(\delta)))\in\App$.
As we will see below, this condition gives some independence of
the traces of $\rho(\gamma)^n$ and $\rho(\delta)^n$
for $n$ large.
\begin{observation}\label{obs:IndGene}
  Let $(\gamma,\delta)\in P$. Then $\Ind(\gamma,\delta)$
  has full measure in~$X$.
\end{observation}
\begin{proof}
  Consider the map
  $\Theta=(\theta_\gamma,\theta_\delta)
  \colon X\to [0,1]^2$.
  We want to show that $\Theta^{-1}(\App)$ has full measure
  in $X$. If $\Theta$ is a submersion
  at $[\rho]$, the implicit theorem implies that
  $\Theta^{-1}(\App)$ has full measure locally around $[\rho]$.
  Hence it suffices to show that $\Theta$ is a submersion in a
  dense Zariski open subset of $X$. 
  Consider the Zariski open set $U=\Theta^{-1}(0,1)^2$: it is
  well-known that $d\theta_\gamma$ and $d\theta_\delta$ are
  smooth non-vanishing forms on $U$. 
  If $\gamma$ and $\delta$ are disjoint, $\Theta$ can be extended
  to a system of action-angle coordinate, which implies
  that $\Theta$ is a submersion everywhere in~$U$, see for
  instance~\cite{Jeffrey-Weitsman}.
  If $\gamma,\delta$ do intersect, then it is
  known that their Poisson bracket does not vanish identically,
  see for instance~\cite[Corollary 5.2]{Charles-Marche}.
  As $X$ is irreducible, it follows that
  $d\theta_\gamma,d\theta_\delta$ are linearly independent in
  a Zariski-open subset of $U$, proving the lemma.
\end{proof}
From the Observation~\ref{obs:IndGene}, it follows that the
set
\[ \Ind = \bigcap_{(\gamma,\delta)\in P}\Ind(\gamma,\delta)\]
has full measure in~$X$.
It is obviously $\Mod(\Sigma_g)$-invariant,
and
for any $[\rho]\in \Ind$ and
any non-separating simple curve $\gamma$,
we have $\theta_\gamma(\rho)\in(0,1)$; in fact
$\theta_\gamma(\rho)$ is irrational.

\subsection{The proof}
Since the action of $\Mod(\Sigma_g)$ on $X$
is ergodic (by Goldman~\cite{Goldman97}), the set
\[ D = \left\lbrace [\rho]\, ; \Mod(\Sigma_g)\cdot[\rho]
\text{ is dense in }X \right\rbrace \]
has full measure in~$X$.
In fact, this set is known explicitely from the work of
Previte and Xia~\cite{PreviteXia}; it is the set of those $[\rho]$
such that the image of $\rho$ is dense in $\SU_2$.
Thus, the set
$D\cap\Ind$ also has full measure, and Theorem~\ref{thm:PresqueMin}
will follow from the following statement.
\begin{proposition}\label{prop:PresqueMin}
  For all $[\rho]\in D\cap\Ind$, the set $\Gamma\cdot[\rho]$ is
  dense in~$X$.
\end{proposition}
The proof resides on the following lemma.
\begin{lemma}\label{lem:PresqueMin}
  Let $\gamma$ be a non-separating simple closed curve and
  $[\rho]\in\Ind$. Then $\tau_\gamma\cdot[\rho]$ is in the
  closure of $\Gamma\cdot[\rho]$.
\end{lemma}
\begin{proof}
  Consider an element $\varphi\in \Gamma$ as in
  Observation~\ref{obs:OnBouge} and set $\delta=\varphi(\gamma)$.
  We observe that for any $n\in \mathbb{N}$,
  $\tau_\gamma^{n}\tau_\delta^{-n}=
  \tau_\gamma^n \varphi \tau_\gamma^{-n}\varphi^{-1}$
  belongs to $\Gamma$.
  
  Write $\alpha=\theta_\delta(\rho)$
  and $\beta=\theta_\gamma(\rho)$.
  As $\alpha\in(0,1)$,
  the twist flow  $(\Phi_\delta^s)_{s\in\R/\Z}$ is well defined on
  $\Phi_\gamma^t([\rho])$ for all $t$ in a
  neighborhood $I$ of $0$ in $\R/\Z$.
  We set $F(t,s)=\Phi_\gamma^{s} \Phi_\delta^{-t}([\rho])$ for
  $(t,s)\in I\times \R/\Z$. From the identity $\tau_\gamma=
  \Phi_\gamma^{\theta_\gamma}$, we get for all $n$ such that $n\alpha\in I$
  \[\tau_\gamma^{n}\tau_\delta^{-n}[\rho]=F(n \alpha,nf(n\alpha))\]
  where $f(t)=\theta_\gamma(\Phi_\delta^{-t}([\rho]))$.
  As $\beta\in(0,1)$, the function $\theta_\gamma$ is smooth
  at $[\rho]$ hence $f$ is smooth at $0$.
  To prove the lemma, it is sufficient to show that one has
  $(n\alpha, nf(n\alpha))\to (0, \beta)$
  for a sequence of $n$'s going to infinity.
  
  Since $(\alpha,\beta)\in\App$, there exists a sequence
  $(q_n)$ of integers as in Definition~\ref{def:App}.
  We have $q_n\alpha\to 0$, so we
  consider the Taylor expansion of $f$ at $0$: since
  $|q_n\alpha| = o(\frac{1}{q_n})$, this gives
  \[ f(q_n\alpha) = f(0) + o(\frac{1}{q_n}), \]
  so $q_nf(q_n\alpha) = q_n\beta + o(1)$. Now, $q_n\beta$ tends to
  $\beta$, by Definition~\ref{def:App}.
\end{proof}

We are ready to conclude the proof of Theorem~\ref{thm:PresqueMin}.
\begin{proof}[Proof of Proposition~\ref{prop:PresqueMin}]
  Recall that $D\cap\Ind$ is $\Mod(\Sigma_g)$-invariant.
  Let $[\rho]\in D\cap\Ind$ and let $\gamma_1,\gamma_2$ be two
  non-separating simple closed curves. By Lemma~\ref{lem:PresqueMin},
  there exists a sequence $(\varphi_n)$ of elements of $\Gamma$
  such that $\varphi_n\cdot[\rho]\to\tau_{\gamma_2}\cdot[\rho]$.
  For all $n$, we may apply Lemma~\ref{lem:PresqueMin} to
  $\varphi_n\cdot[\rho]$,
  and now we can apply a diagonal argument to show that
  $\tau_{\gamma_1}\tau_{\gamma_2}\cdot[\rho]$ is in the closure
  of~$\Gamma\cdot[\rho]$.
  %
  We
  proceed by induction:
  for all $[\rho]\in D\cap\Ind$,
  and all curves $\gamma_1,\ldots,\gamma_n$, the representation
  $\tau_{\gamma_1}\cdots\tau_{\gamma_n}\cdot[\rho]$ is in the
  closure of~$\Gamma\cdot[\rho]$.
  
  We notice\footnote{
  alternatively, we may use the beautiful fact that
  $\Mod(\Sigma_g)$ is {\em positively} generated by
  Dehn twists,
  see~\cite[Paragraph~5.1.4]{FarbMargalit}.}
  that Lemma~\ref{lem:PresqueMin} works equally well
  for $\tau_\gamma^{-1}$ instead of $\tau_\gamma$, hence we
  deduce that the whole orbit
  $\Mod(\Sigma_g)\cdot[\rho]$ (and hence,
  also its closure)
  is contained in the closure of $\Gamma\cdot[\rho]$. As
  $[\rho]\in D$, this implies that $\Gamma\cdot[\rho]$ is dense
  in~$X$.
\end{proof}

\bibliographystyle{plain}

\bibliography{Biblio}

\end{document}